\documentclass[11pt,twoside]{amsart}
	\usepackage{amssymb}
\usepackage{graphicx,color}
\usepackage{graphicx}			
\usepackage[all,cmtip]{xy}

\theoremstyle{definition}
\newtheorem{thm}{Theorem}[section]

\newtheorem{defi}[thm]{Definition}
\newtheorem{exm}[thm]{Example}
\newtheorem{lem}[thm]{Lemma}
\newtheorem{rem}[thm]{Remark}
\newtheorem{pps}[thm]{Proposition}

\newtheorem{quest}[thm]{Question}

\newcommand{\CC}{\mathbb{C}}

\newcommand{\KK}{\mathbb{K}}
\newcommand {\PP}{\mathbb{P}}

\newcommand {\EE}{\mathbb{E}}

\newcommand{\cA}{\mathcal{A}}

\newcommand{\cM}{\mathcal{M}}

\newcommand{\cO}{\mathcal{O}}
\newcommand{\cE}{\mathcal{E}}

\newcommand{\nii}{\noindent}

\DeclareMathOperator{\rk}{rank}

\begin{document}

\title{Irreducibility of the moduli space of orthogonal instanton bundles on $\PP^n$}
\author[Aline V. Andrade]{Aline V. Andrade}
\address{Instituto de Matem\'atica, Estat\'istica e Computa\c{c}\~{a}o Cient\'ifica - UNICAMP, Rua S\'ergio Buarque de Holanda 651, Distr. Bar\~ao Geraldo, CEP 13083-859, Campinas (SP), Brasil}
\email{aline.andrade@ime.unicamp.br}

\author[Simone Marchesi]{Simone Marchesi}
\address{Instituto de Matem\'atica, Estat\'istica e Computa\c{c}\~{a}o Cient\'ifica - UNICAMP, Rua S\'ergio Buarque de Holanda 651, Distr. Bar\~ao Geraldo, CEP 13083-859, Campinas (SP), Brasil}
\email{marchesi@ime.unicamp.br}

\author[Rosa M. Mir\'o-Roig]{Rosa M. Mir\'o-Roig}
\address{Department de matem\`{a}tiques i Inform\`{a}tica, Universitat de Barcelona, Gran Via de les Corts Catalanes 585, 08007 Barcelona,
Spain}
\email{miro@ub.edu}

\begin{abstract}
In order to obtain existence criteria for orthogonal instanton bundles on $\PP^n$, we provide a bijection between equivalence classes of orthogonal instanton bundles with no global sections and symmetric forms. Using such correspondence we are able to provide explicit examples of orthogonal instanton bundles with no global sections on $\PP^n$ and prove that every orthogonal instanton bundle with no global sections on $\PP^n$ and charge $c\geq 3$ has rank $r \leq (n-1)c$. We also prove that when the rank $r$ of the bundles reaches the upper bound, $\cM_{\PP^n}^{\cO}(c,r)$, the coarse moduli space of orthogonal instanton bundles with no global sections on $\PP^n$, with charge $c\geq 3$ and rank $r$,  is affine, reduced and irreducible. Last, we construct Kronecker modules to determine the splitting type of the bundles in $\cM_{\PP^n}^{\cO}(c,r)$, whenever is non-empty.\end{abstract}

\thanks{ The first author was supported by CAPES process number 99999.000282/2016-02. The second author is partially supported by Fapesp grant 2017/03487-9. The third author has been partially supported by MTM2016-78623-P.
\\ {\it 2010 Mathematic Subject Classification.} 14D20 and 14J60.}

\maketitle

\tableofcontents

\markboth{Aline V. Andrade, Simone Marchesi, Rosa M. Mir\'o-Roig}{Irreducibility of the moduli space of orthogonal instanton bundles on $\PP^n$}

\today

\large

\section{Introduction}

\nii Since the 1970's the "instantons" or pseudo-particle solutions of the classical Yang-Mills equations in the Euclidean $4$-space have awaken great interest in the physical and mathematical communities due the link that they provide between algebraic geometry and mathematical physics (see for instance \cite{AHS1977} and \cite{AW1977}). In \cite{ADHM1978} Atiyah, Drinfeld, Hitchin and Manin provided the so called "ADHM contruction of instanton" on $\PP^3$. In \cite{OS1986} Okonek and Spindler generalized the contruction of instanton bundles to $\PP^{2n+1}$, since then the study of this family of bundles and their moduli spaces have been a central topic in algebraic geometry. The moduli space $\cM_{\PP^3}(c)$ of the $c$-instanton bundles on $\PP^3$, i.e. of stable $2-$bundles $\cE$ with Chern classes $(c_1,c_2)=(0,c)$ and $\textrm{H}^1(\cE(-2))=0$, is expected to be a smooth and irreducible variety with dimension $8c-3$ for $c\geq 1$. This problem was approached by several authors (see \cite{Bar1981}, \cite{CTT2003}, \cite{ES1981},  \cite{Har1978}, \cite{KO2003}, \cite{LeP1983}) and while the irreducibility was completely solved by Tikhomirov in \cite{Tik2012} and \cite{Tik2013}, the smoothness was solved on $\CC\PP^3$ by Jardim and Verbitsky (see \cite{JV2014}), but remains open for the $3$-dimensional projective space over any other algebraically closed field of characteristic $0$.

\nii In order to understand moduli spaces of stable vector bundles over a projective variety, in \cite{Jar2005} Jardim extended the definition of instantons to even-dimensional projective spaces and allowed non-locally-free sheaves of arbitrary rank. Jardim defined an {\bf instanton sheaf on $\PP^n$ ($n \geq 2$)} as a torsion-free coherent sheaf $\cE$ on $\PP^n$ with first Chern class $c_1(\cE) = 0$ satisfying some cohomological conditions (see Definition \ref{definst} for details). If $\cE$ is locally-free, $\cE$ is called an {\bf instanton bundle} and in addition, if $\cE$ is a rank-$2n$ bundle on $\PP^{2n+1}$ with trivial splitting type, then $\cE$ is a {\bf mathematical instanton bundle} as defined by Okonek and Spindler in \cite{OS1986}. Studying the moduli space of instanton bundles becomes more complicated for higher dimensional projective spaces or higher rank, because of this many authors have considered instanton bundles with some additional structure (special, symplectic and orthogonal). For example, in \cite{CHMS2014} Costa, Hoffmann, Mir\'o-Roig and Schmitt proved that the moduli space of all symplectic instanton bundles on $\PP^{2n+1}$ with $n\geq 2$ is reducible; in \cite{MO1997} Mir\'o-Roig and Orus-Lacort proved that $\cM_{\PP^{2n+1}}(c)$ is singular for $n \geq 2$ and $c\geq 3$; Costa and Ottaviani in \cite{CO12002} proved that $\cM_{\PP^{2n+1}}(c)$ is affine and introduced an invariant which allowed Farnik, Frapporti and Marchesi to prove in \cite{FFM2009} that there are no orthogonal instanton bundles with rank $2n$ on $\PP^{2n+1}$. Using the ADHM construction introduced by Henni, Jardim and Martins in \cite{HJV2015}, Jardim, Marchesi and Wi\ss{}dorf in \cite{JMW2016} consider autodual instantons of arbitrary rank on projective spaces, with focus on symplectic and orthogonal instantons; they described the moduli space of framed autodual instanton bundles and showed that there are no orthogonal instanton bundles with trivial splitting type, arbitrary rank $r$ and charge $2$ or odd on $\PP^n$.

\nii While in \cite{AB2015} Abuaf and Boralevi proved that the moduli space of rank $r$ stable orthogonal bundles on $\PP^2$, with Chern classes $(c_1,c_2)=(0,c)$ and trivial splitting type on the general line, is smooth and irreducible for $r=c$ and $c\geq 4$, and $r=c-1$ and $c\geq 8$, the results of Farnik, Frapport and Marchesi in \cite{FFM2009} and Jardim, Marchesi and Wi\ss{}dorf in \cite{JMW2016}, already mentioned, show us that orthogonal instanton bundles on $\PP^n$, $n\geq 3$ are for some reason hard to find and that it is interesting to establish existence criteria for these bundles.

\nii Hence the main goal of this work is to provide existence criteria for orthogonal instanton bundles with higher rank on $\PP^n$, for $n\geq 3$ and then to study their moduli space and splitting type.\\

\nii Next, we outline the structure of the paper. In section 2 we introduce some preliminaries necessary through the text. In section 3, in order to establish existence criteria for orthogonal instanton bundles on $\PP^n$, for $n\geq 3$, we define certain equivalence classes of orthogonal instanton bundles and provide a bijection between these classes and symmetric forms (see Theorem \ref{thmequivalence}). Using such correspondence we prove the following result.\\

\nii {\bf Proposition} \ref{mainthm} Let $c\geq 3$ be an integer. Every orthogonal instanton bundle $\cE$ on $\PP^n$ ($n\geq 3$) with no global sections and charge $c$ has rank $r\leq(n-1)c$. Moreover, there are no orthogonal instanton bundles $\cE$ on $\PP^n$ with no global sections and charge $c$ equal to $1$ or $2$.\\

\nii In section 4, we study the case when the rank reaches the upper bound and prove that there exists an affine coarse moduli space for our problem and addition we prove that this moduli space is irreducible and reduced (see Theorem \ref{finemoduli}). Finally, in section 5, given an orthogonal instanton bundle $\cE$ on $\PP^n$ with charge $c$, rank $(n-1)c$ and no global sections, for $c,n\geq 3$, we construct a Kronecker module to determine whether the restriction of $\cE$ to a line $L\subset \mathbb{P}^n$ is trivial or not (see Theorem \ref{splitting}).

\section{Preliminaries}

\nii Let $\KK$ be an algebraically closed field of characteristic $0$. Let us consider $\PP^n = \PP(V)$, where $V$ is a $(n+1)$-dimensional $\KK$-vector space,  $n\geq 2$. If $\cE$ is a vector bundle on $\PP^n$, then $h^i(\cE(k))$ denotes the dimension of $\textrm{H}^i(\cE(k))$, the $i^{th}$ cohomology group of $\cE$, and $\cE^{\vee}$ denotes the dual of $\cE$, i. e., $\cE^{\vee}=\mathcal{H}om(\cE, \cO_{\PP^n})$. We denote by $H_c$ a $c$-dimensional $\KK$-vector space, with $c\geq 1$ and if $U$ is a $\KK$-vector space, we denote by $U^{\vee}$ the dual vector space of $U$. \\

\begin{defi}\label{definst} An {\bf instanton }sheaf on $\PP^n$  is a torsion-free coherent sheaf $\mathcal{E}$ on $\PP^n$ with $c_1(\cE)=0$ satisfying the following cohomological conditions:
\begin{itemize}
								\item[(1)] $\textrm{H}^0(\mathcal{E}(-1))=\textrm{H}^n(\mathcal{E}(-n))=0$;
								\item[(2)] $\textrm{H}^1(\mathcal{E}(-2))=\textrm{H}^{n-1}(\mathcal{E}(1-n))=0$, if $n\geq 3$;
								\item[(3)] $\textrm{H}^i(\mathcal{E}(-k))=0$ for $2 \leq p \leq n-2$ and all $k$, if $n\geq 4$.
							\end{itemize}
The integer $c=-\chi(\mathcal{E}(-1)) $ is called charge of $\mathcal{E}$.
\end{defi}



\nii We will say that a vector bundle $\mathcal{E}$ is {\bf autodual} if it is isomorphic to its dual, i.e. there exists an isomorphism $\phi: \mathcal{E} \rightarrow \mathcal{E}^{\vee}$. If the isomorphism $\phi$ satisfies $\phi^{\vee}=-\phi$, the vector bundle is called {\bf symplectic}. If the isomorphism $\phi$ satisfies $\phi^{\vee}=\phi$, the vector bundle is called {\bf orthogonal}.\\

\nii Let $\mathcal{E}$ be an {\bf orthogonal instanton bundle over $\PP^n$ ($n\geq3$)} with charge $c$, rank $r$ and no global sections ($\textrm{H}^0(\mathcal{E})=0$). Considering the following exact sequence
\begin{equation*}
\displaystyle \xymatrix{0\ar[r] & \mathcal{E}(-i-1) \ar[r] & \mathcal{E}(-i) \ar[r] & \mathcal{E}(-i)\mid_{\PP^{n-1}} \ar[r] & 0},
\end{equation*}
\nii by the instanton cohomological conditions in Definition \ref{definst}, the Serre duality and the Hirzebruch-Riemann-Roch theorem, for $0\leq i \leq n$ and $-n-1\leq k \leq 0$, one has 

\begin{center}
$h^i(\mathcal{E}(k))=\left\lbrace\begin{array}{ll}
c,& \textrm{if }(i,k) \in\{(1,-1),(n-1,-n)\};\\
(n-1)c-r,&\textrm{if }(i,k)\in\{(1,0),(n-1,-n-1)\};\\
0, &\textrm{otherwise.}
\end{array}\right.$
\end{center}
\vspace{2mm}

\section{The equivalence}

\nii Consider a triple $(\cE,\phi, f)$, where
\begin{itemize}
\item $\cE$ is an orthogonal instanton bundle on $\PP^n$ with charge $c$, rank $r$ and no global sections.
\item $\xymatrix{\phi:\cE\ar[r]^-{\cong}& \cE^{\vee} }$ is an orthogonal structure of $\cE$, i.e. $\phi^{\vee}=\phi$.
\item $\xymatrix{f:H_c \ar[r]^-{\cong}& \textrm{H}^{n-1}(\cE(-n)) }$.
\end{itemize}

\begin{defi}\label{defeq} Two triples $(\cE_1, \phi_1, f_1)$ and $(\cE_2, \phi_2 ,f_2)$ are called {\bf equivalent} if there is an isomorphism $\xymatrix{g:\cE_1 \ar[r]^-{\cong}& \cE_2 }$ such that  the following diagrams commute

\begin{equation*}
\xymatrix{ \cE_1 \ar[r]^-{\phi_1} \ar[d]_-{g} & \cE_1^{\vee}&  &H_c \ar[r]^-{f_1} \ar[d]_-{\lambda \textrm{Id}_c}& \textrm{H}^{n-1}(\cE_1(-n))\ar[d]^-{g_{\ast}}\\
\cE_2 \ar[r]_-{\phi_2} & \cE_2^{\vee} \ar[u]_-{g^{\vee}}&  &H_c \ar[r]_-{f_2}& \textrm{H}^{n-1}(\cE_2(-n)),
}
\end{equation*}

\nii where $\xymatrix{g_{\ast}:\textrm{H}^{n-1}(\cE_1(-n)) \ar[r]^-{\cong}& \textrm{H}^{n-1}(\cE_2(-n)) }$ is the induced isomorphism in cohomology and $\lambda \in \{-1,1\}$. We denote by $[\cE, \phi, f]$ the {\bf equivalence class} of a triple $(\cE,\phi,f)$.
\end{defi}

\nii Fixing the integers $c$ and $r$, we will denote by $\EE[c,r]$ the set of all equivalence classes $[\cE,\phi, f]$ of orthogonal instanton bundles with charge $c$, rank $r$ and no global sections over $\PP^n$.

\begin{lem}\label{help} Each triple $(\cE,\phi, f) \in \EE[c,r]$ defines a morphism $\xymatrix{A:\textrm{H}_c \otimes V \ar[r]& \textrm{H}_c^{\vee} \otimes V^{\vee}}$, living in $ \bigwedge^2 H_c^{\vee}\otimes \bigwedge^2 V^{\vee}$, that in turn defines a monad.

\end{lem} 

\begin{proof}

\nii We consider the Euler exact sequence and its exterior powers

\begin{equation}\label{Eulerseq}
\xymatrix {0 \ar[r] & \Omega_{\PP^n}^1\ar[r]^-{i_1} & V^{\vee} \otimes \cO_{\PP^n}(-1) \ar[r]^-{ev} &\cO_{\PP^n} \ar[r]& 0},
\end{equation}

\begin{equation}\label{koszuli}
\xymatrix {0 \ar[r] & \Omega^{i+1}_{\PP^n}\ar[r] &\wedge^{i+1} V^{\vee} \otimes \cO_{\PP^n}(-i-1) \ar[r] &\Omega_{\PP^n}^i \ar[r]& 0}, \text{ with $1\leq i\leq n-2$, and}
\end{equation}

\begin{equation}\label{koszul}\xymatrix {0 \ar[r] & \bigwedge^{n+1} V^{\vee} \otimes \cO_{\PP^n}(-n-1) \ar[r] &\bigwedge^n V^{\vee} \otimes \cO_{\PP^n}(-n) \ar[r]^-{i_2} &\Omega^{n-1}_{\PP^n} \ar[r]& 0}
\end{equation}

\nii induced by the Koszul complex of $\xymatrix{V^{\vee}\otimes\cO_{\PP^{n}}(-1)\ar[r]^-{ev}&\cO_{\PP^{n}} }$, where $ev$ denotes the canonical evaluation map. Tensoring \eqref{Eulerseq} with $\cE$ we obtain $\textrm{H}^i(\cE \otimes \Omega_{\PP^n}^1)=0 $, for $i=0,3,\ldots,n$ and the exact sequence
\begin{equation}\label{omegai}
\xymatrix {0 \ar[r] & \textrm{H}^1(\cE \otimes \Omega_{\PP^n}^1)\ar[r]^-{i_1} & \textrm{H}^1(V^{\vee} \otimes \cE(-1)) \ar[r] &\textrm{H}^1(\cE) \ar[r]& H^2(\cE \otimes \Omega_{\PP^n}^1)\ar[r] & 0;}
\end{equation}

\nii Tensoring \eqref{koszuli} with $\cE$ we obtain
\begin{equation}\label{omegaii}
\textrm{H}^j(\cE \otimes \Omega_{\PP^n}^i)\cong \textrm{H}^{j+1}(\cE \otimes \Omega_{\PP^n}^{i+1})
\end{equation}
\nii and $\textrm{H}^0(\cE \otimes \Omega_{\PP^n}^i)= \textrm{H}^{n-1}(\cE \otimes \Omega_{\PP^n}^{i+1})=0 $, for $1\leq i\leq n-2$ and $0\leq j \leq n-1$. Finally tensoring \eqref{koszul} with $\cE$ we obtain $\textrm{H}^i(\cE \otimes \Omega_{\PP^n}^{n-1})=0 $, for $i=0,\ldots,n-3, n$ and the exact sequence

\begin{equation}\label{omegaiii}
0 \rightarrow \textrm{H}^{n-2}(\cE \otimes \Omega^{n-1}_{\PP^n})\rightarrow \textrm{H}^{n-1}(\bigwedge^{n+1} V^{\vee} \otimes \cE(-n-1)) \rightarrow \textrm{H}^{n-1}(\bigwedge^n V^{\vee} \otimes \cE(-n)) \stackrel{i_2}{\rightarrow} \textrm{H}^{n-1}(\cE \otimes \Omega^{n-1}_{\PP^n}) \rightarrow 0.
\end{equation}

\nii Therefore we have
$$\textrm{H}^2(\cE \otimes \Omega_{\PP^n}^1)=\textrm{H}^3(\cE \otimes \Omega_{\PP^n}^2)=\cdots=\textrm{H}^n(\cE \otimes \Omega_{\PP^n}^{n-1})=0, $$

$$\textrm{H}^{n-2}(\cE \otimes \Omega_{\PP^{n}}^{n-1})=\textrm{H}^{n-3}(\cE \otimes \Omega_{\PP^n}^{n-2})=\cdots=\textrm{H}^0(\cE \otimes \Omega_{\PP^n}^{1})=0, $$

\begin{equation*}
h^{n-1}(\cE \otimes \Omega^{n-1}_{\PP^n})= h^1(\cE \otimes \Omega^1_{\PP^n})=h^1(V^{\vee}\otimes \cE(-1))-h^1(\cE)=2c+r,
\end{equation*}

\nii and  the exact sequences
\begin{equation*}
\small{\xymatrix{0 \ar[r] & \textrm{H}^{n-1}(\cE(-n-1)) \otimes\bigwedge^{n+1} V^{\vee} \ar[r]^-{a} & \textrm{H}^{n-1}(\cE(-n)) \otimes \bigwedge^n V^{\vee} \ar[r]^-{i_2} &\textrm{H}^{n-1}(\cE \otimes \Omega^{n-1}_{\PP^n}) \ar[r] & 0}}
\end{equation*}
\nii and
\begin{equation}\label{cttA}
\xymatrix{0 \ar[r] & \textrm{H}^1(\cE \otimes \Omega^1_{\PP^n}) \ar[r]^-{i_1} & \textrm{H}^1(\cE(-1)) \otimes V^{\vee}\ar[r]^-{b} &\textrm{H}^1(\cE) \ar[r] & 0.}
\end{equation}

\nii By the functoriality of the Serre duality, we have $i_1=i_2^{\vee}$ and  the following diagram with exact rows

\begin{equation}\label{Aprime}
\tiny{\xymatrix{0 \ar[r] & \textrm{H}^{n-1}(\cE(-n-1)) \otimes\bigwedge^{n+1} V^{\vee} \ar[r]^-{a} & \textrm{H}^{n-1}(\cE(-n)) \otimes \bigwedge^n V^{\vee} \ar[d]^-{A'}\ar[r]^-{i_2} &\textrm{H}^{n-1}(\cE \otimes \Omega^{n-1}_{\PP^n}) \ar[r] & 0\\
0 & \textrm{H}^1(\cE)  \ar[l] & \textrm{H}^1(\cE(-1)) \otimes V^{\vee}\ar[l]^-{b}&\textrm{H}^1(\cE \otimes \Omega^1_{\PP^n}) \ar[l]^-{i_2^{\vee}}\ar[u]_-{\cong}^-{\partial} & \ar[l]0}}
\end{equation}

\nii where $A'=i_2^{\vee} \circ \partial^{-1} \circ i_2$, moreover, the Euler  sequence \eqref{Eulerseq} yields the canonical isomorphism $\small{\xymatrix{\omega_{\PP^n} \ar[r]^-{\cong}& \bigwedge^{n+1} V^{\vee} \otimes \cO_{\PP^n}(-n-1)}}$. So, fixing an isomorphism $ \xymatrix{\tau:\KK \ar[r]^-{\cong}& \bigwedge^{n+1} V^{\vee}}$, we have the isomorphisms
\begin{equation}\label{tau12}
\xymatrix{\tau_1:V \ar[r]^-{\cong}& \bigwedge^{n} V^{\vee}} \textrm{ and } \xymatrix{\tau_2: \omega_{\PP^n} \ar[r]^-{\cong}& \cO_{\PP^n}(-n-1)}.
\end{equation}

\nii Thus, each  $[\cE, f, \phi]\in \EE(c,r)$ defines a morphism $\xymatrix{A:\textrm{H}_c \otimes V \ar[r]& \textrm{H}_c^{\vee} \otimes V^{\vee}}$ through the following composition

$$\tiny{\xymatrix{A:H_c\otimes V \ar[r]^-{\textrm{Id} \otimes \tau_{1}}& H_c\otimes \bigwedge^n V^{\vee} \ar[r]^-{f\otimes \textrm{Id}}& \textrm{H}^{n-1}(\cE(-n))\otimes \bigwedge^n V^{\vee}\ar[r]^-{A'}& \textrm{H}^1(\cE(-1))\otimes V^{\vee}\ar[r]^-{\phi\otimes \textrm{Id}}& \textrm{H}^1(\cE^{\vee}(-1))\otimes V^{\vee}\\
& \ar[r]^-{\textrm{SD}\otimes \textrm{Id}}& \textrm{H}^{n-1}(\cE(1)\otimes \omega_{\PP^n})^{\vee}\otimes V^{\vee} \ar[r]^-{\tau_{2}\otimes \textrm{Id}} & \textrm{H}^{n-1}(\cE(-n))^{\vee}\otimes V^{\vee}\ar[r]^-{f^{\vee}\otimes \textrm{Id}}&H^{\vee}_c \otimes V^{\vee},
}}$$

\nii where SD denotes the Serre duality isomorphism. Therefore we can write

\begin{equation}\label{Adef}
A= ((f^{\vee}\circ \tau_2 \circ \textrm{SD} \circ \phi) \otimes \textrm{Id})\circ A^{\prime}\circ (f \otimes \tau_1).
\end{equation}

\nii Note that, since $\tau$ is a multiplication by a scalar, $A$ does not depend on the choice of $\tau$.

\nii It is possible to prove that $A$ is symmetric, therefore,

\begin{equation}\label{imageA}
A \in (S^2 H^{\vee}_c \otimes S^2 V^{\vee}) \oplus (\bigwedge^2H^{\vee}_c \otimes \bigwedge^2 V^{\vee}).
\end{equation}

\nii We will show that actually $A \in \bigwedge^2 H^{\vee}_c \otimes \bigwedge^2 V^{\vee}$.
\nii Let us set $W:=\dfrac{H_c\otimes V}{\textrm{Ker } A}$, and combining \eqref{Aprime} and \eqref{Adef} we have the commutative diagram of exact rows

\begin{equation}\label{DiagA}
\small{\xymatrix{
 &  & H_c \otimes V \ar@{-->}@/_1cm/[ddd]_A \ar[d]^{f \otimes \tau_1}& & \\
0 \ar[r] & H^{n-1}(\cE(-n-1)) \otimes\bigwedge^{n+1} V^{\vee} \ar[r]^-{a} & H^{n-1}(\cE(-n)) \otimes \bigwedge^n V^{\vee} \ar[d]^{A^{\prime}}\ar[r]^-{i_2} &H^{n-1}(\cE \otimes \Omega^{n-1}_{\PP^n)}) \ar[r] & 0\\
0 & H^1(\cE)  \ar[l] & H^1(\cE(-1)) \otimes V^{\vee}\ar[l]^-{b} \ar[d]^{(f^{\vee}\circ \tau_2 \circ SD \circ \phi) \otimes \textrm{Id}}&H^1(\cE \otimes \Omega_{\PP^n)}) \ar[l]^-{i_2^{\vee}}\ar[u]_{\cong}^-{\partial} & \ar[l]0,\\
 &  & H_c^{\vee} \otimes V^{\vee} & & }}
\end{equation}

\nii which tells us that $\textrm{dim } W =2c+r$ and induces the diagram

\begin{equation}\label{defqA}
\xymatrix{0 \ar[r] & \textrm{Ker } A \ar[r] & H_c\otimes V \ar[d]^{A}\ar[r]^-{p} &W \ar[r] \ar[d]_-{\cong}^-{q_A}& 0\\
0 & \textrm{Ker } A^{\vee}  \ar[l] & H_c^{\vee}\otimes V^{\vee} \ar[l]&W^{\vee}\ar[l]^-{p^{\vee}} & \ar[l]0.}
\end{equation}

\nii where $p$ is the canonical projection and $\xymatrix{q_A: W \ar[r]^-{\cong}& W^{\vee}}$ is a symmetric isomorphism. So we can define the induced morphism of sheaves

\begin{equation}
\xymatrix{a^{\vee}_A: W^{\vee} \otimes \cO_{\PP^n} \ar[r]^-{p^{\vee}\otimes \textrm{Id}}&H^{\vee}_c \otimes V^{\vee}\otimes \cO_{\PP^n} \ar[r]^-{\textrm{Id}\otimes ev}& H^{\vee}_c \otimes \cO_{\PP^n}(1)}
\end{equation}

\nii which is surjective, therefore $a_A$ is injective, and the composition $$\xymatrix{\psi: H_c \otimes \cO_{\PP^n}(-1)\ar[r]^-{a_A}&  W \otimes \cO_{\PP^n} \ar[r]^{q_A\otimes \textrm{Id}} & W^{\vee} \otimes \cO_{\PP^n} \ar[r]^-{a_A^{\vee}} & H^{\vee}_c \otimes \cO_{\PP^n}(1)}$$

\nii is zero.

\nii Since $A$ is symmetric, we can write $A=A_1+A_2$, where $A_1\in \bigwedge^2 H_c^{\vee}\otimes \bigwedge^2 V^{\vee}$ and $A_2\in S^2 H_c^{\vee}\otimes S^2 V^{\vee}$. By the Euler sequence \eqref{Eulerseq} we have

\begin{equation}\label{eval}
\xymatrix{0\ar[r]&\bigwedge^2(\Omega(1))\ar[r]& \bigwedge^2 V^{\vee}\otimes \cO\ar[rr]^-{( \textrm{Id}\otimes ev)\circ(i \otimes \textrm{Id})}&&V^{\vee}\otimes\cO(1)\ar[r]^-{ev}&\cO(2)\ar[r]&0,}
\end{equation}

\nii where $i:\bigwedge^2V^{\vee}\hookrightarrow V^{\vee}\otimes V^{\vee}$ is the inclusion. Note that $\psi= (\textrm{Id}\otimes ev) \circ A \circ (\textrm{Id}\otimes ev^{\vee})$, thus
$$\psi= (\textrm{Id}\otimes ev) \circ A_1 \circ (\textrm{Id}\otimes ev^{\vee})+(\textrm{Id}\otimes ev) \circ A_2 \circ (\textrm{Id}\otimes ev^{\vee}).$$

\nii By the sequence \eqref{eval}, we have $\textrm{Im } A_1 \subset \textrm{Ker } ev$ and therefore $\psi=(\textrm{Id}\otimes ev) \circ A_2 \circ (\textrm{Id}\otimes ev^{\vee})$; moreover, $\textrm{Im } A_2 \subset S^2H_c^{\vee}\otimes S^2V^{\vee}\not\subset \textrm{Ker } ev$, otherwise the evaluation map would be the zero map. Hence, $\psi=0$ implies $A_2=0$ and therefore $A \in \bigwedge^2 H_c^{\vee}\otimes \bigwedge^2 V^{\vee}$.

\nii On the other hand, for each $A \in \bigwedge^2 H_c^{\vee} \otimes \bigwedge^2 V^{\vee}$, it follows from the previous paragraph that $(\textrm{Id}\otimes ev) \circ A \circ (\textrm{Id}\otimes ev^{\vee})=0$, therefore we can associate the monad

\begin{equation}\label{AMonad}
\xymatrix{\cM_A: H_c \otimes \cO_{\PP^n}(-1)\ar[r]^-{a_A}&  W \otimes \cO_{\PP^n}  \ar[rr]^-{a_A^{\vee}\circ (q_A\otimes \textrm{Id})} & &H^{\vee}_c \otimes \cO_{\PP^n}(1),}
\end{equation}
\nii whose cohomology sheaf is defined by

\begin{equation}\label{EAdef}
\cE_A:=\dfrac{\textrm{Ker }(a_A^{\vee}\circ (q_A\otimes Id))}{\textrm{Im } a_A}.
\end{equation}

\end{proof}

Recall that asking $\cE$ not to have global sections is equivalent to not having trivial summands in the vector bundle $\textrm{Ker }(a_A^{\vee}\circ (q_A\otimes Id))$ (see \cite{Arr2010} for more details).\smallskip\\

\nii Recall also that $A$ is called {\bf non-degenerate} if $A(h \otimes v) \neq 0$ for any non-zero decomposable tensor $h\otimes v\in H_c \otimes V$. Hence, similar to \cite{CTT2003}, with the notation of the previous proof, the followings are equivalent:

\begin{itemize}
\item[(i)] $a_A^{\vee}\circ (q_A\otimes \textrm{Id})$ is surjective;
\item[(ii)] the image of $a_A$ is a subbundle;
\item[(iii)] $A$ is non-degenerate.
\end{itemize}

\nii From all the previous observations, the map $A$ defined in \eqref{Adef} has the following properties:

\begin{itemize}
\item[(A1)] $\textrm{rank } (A: H_c\otimes V \rightarrow H_c^{\vee}\otimes V^{\vee}) = 2c+r$;
\item[(A2)] $A$ is non degenerate;
\item[(A3)] there exists a symmetric isomorphism $\xymatrix{q_A: W \ar[r]^-{\cong}& W^{\vee}}$, where $W=\dfrac{H_c\otimes V}{\textrm{Ker } A}$.
\end{itemize}

\nii Consider the set
\begin{equation*}
\cA[c,r] :=\left\{  A \in \bigwedge^2 H_c^{\vee}\otimes \bigwedge^2V^{\vee} ; \textrm { such that (A1)-(A3) holds}\right\},
\end{equation*}
\nii our next goal is to prove a bijection between the sets $\cA[c,r]$ and $\EE[c,r]$. To do so, we will need the next result.

\begin{lem}\label{ctt}
For any $A\in \cA[c,r]$, there are isomorphisms
\begin{equation*}
\begin{array}{ccc}
H_c\cong\textrm{H}^{n-1}(\cE_A\otimes \Omega^{n}(1))&W\cong \textrm{H}^1(\cE_A\otimes \Omega^{1}) &\textrm{Ker }A^{\vee}\cong \textrm{H}^1(\cE_A)\\
H_c^{\vee}\cong \textrm{H}^1(\cE_A(-1))&W^{\vee}\cong \textrm{H}^{n-1}(\cE_A\otimes \Omega^{n-1})&
\end{array}
\end{equation*}

\nii which are compatible with the Serre duality and the orthogonal structure $\cE_A\cong \cE_A^{\vee}$, and give the following commutative diagram

\begin{equation*}
{\tiny \xymatrix{H_c\otimes V\ar[r]\ar[d]^{\cong}& W^{\vee}\ar[d]^{\cong}& W\ar[l]_-{q_A}\ar[r]\ar[d]^{\cong}& H_c^{\vee}\otimes V^{\vee}\ar[r]\ar[d]^{\cong}& \textrm{Ker }A^{\vee}\ar[r]\ar[d]^{\cong}&0\\
\textrm{H}^{n-1}(\cE_A\otimes \Omega^{n}(1))\otimes V \ar[r]& \textrm{H}^{n-1}(\cE_A\otimes \Omega^{n-1})& \textrm{H}^1(\cE_A\otimes \Omega^{1})\ar[l]^-{\cong}\ar[r]& \textrm{H}^1(\cE_A(-1))\otimes V^{\vee}\ar[r]& \textrm{H}^1(\cE_A)\ar[r]&0} }
\end{equation*}
\end{lem}
\begin{proof}
\nii Given $A\in \cA[c,r]$ we have the monad
\begin{equation}\label{mon1}\xymatrix{\cM_A: H_c \otimes \cO_{\PP^n}(-1)\ar[r]^-{a_A}&  W \otimes \cO_{\PP^n}  \ar[rr]^-{a_A^{\vee}\circ (q_A\otimes \textrm{Id})} & &H^{\vee}_c \otimes \cO_{\PP^n}(1),}
\end{equation}
\nii whose cohomology bundle is $\cE_A$. On the other hand, applying the Beilinson spectral sequence to $\cE_A(-1)$, one has the monad
$$\small{\xymatrix{0\ar[r]&\textrm{H}^1(\cE_A(1)\otimes \Omega_{\PP^n}^2)\otimes \cO_{\PP^n}(-2)\ar[r]^-{d_1^{-2,1}} & \textrm{H}^1(\cE_A\otimes \Omega_{\PP^n}^1)\otimes \cO_{\PP^n}(-1)\ar[r]^-{d_1^{-1,1}}&\textrm{H}^1(\cE_A(-1))\otimes \cO_{\PP^n}\ar[r]& 0}},$$

\nii and tensoring this monad by $\cO_{\PP^n}(1)$, we obtain the monad
\begin{equation}\label{moncEA}
\small{\xymatrix{0\ar[r]&\textrm{H}^1(\cE_A(1)\otimes \Omega_{\PP^3}^2)\otimes \cO_{\PP^n}(-1)\ar[r]^-{d_1^{-2,1}} & \textrm{H}^1(\cE_A\otimes \Omega_{\PP^n}^1)\otimes \cO_{\PP^n}\ar[r]^-{d_1^{-1,1}}&\textrm{H}^1(\cE_A(-1))\otimes \cO_{\PP^n}(1)\ar[r]& 0}},
\end{equation}
\nii whose cohomology is isomorphic to $\cE_A$.\\

\nii Obviously, $\cE_A\cong \cE_A$ thus we have the isomorphism of the monads \eqref{mon1} and \eqref{moncEA}, which gives us the isomorphisms $H_c\cong\textrm{H}^{1}(\cE_A\otimes \Omega^{2}(1))\cong \textrm{H}^{n-1}(\cE_A\otimes \Omega^{n}(1))$, $W\cong \textrm{H}^1(\cE_A\otimes \Omega^{1})$ and $H_c^{\vee}\cong \textrm{H}^1(\cE_A(-1))$. By Serre duality, we have $W^{\vee}\cong \textrm{H}^1(\cE_A\otimes \Omega^{1})^{\vee}\cong \textrm{H}^{n-1}(\cE_A\otimes \Omega^{n-1})$, and the last isomorphism follows from \eqref{defqA} and \eqref{cttA}.

\nii Finally, the commutativity of the diagram follows from the functoriality of Serre-duality.
\end{proof}

\nii Thanks to the previous lemma, we have the following result.

\begin{thm}\label{thmequivalence}	There exists a bijection between the equivalence classes $[\cE, \phi, f] \in \EE[c,r]$ of orthogonal instanton bundles of charge $c$, rank $r$, with no global sections on $\PP^n $ ($n\geq 3$) and the elements $A \in \cA[c,r]$.
\end{thm}

\begin{proof}
\nii By Lemma \ref{help} given an equivalence class $[\cE, \phi, f]\in \EE[c,r]$, there exists $A \in \bigwedge^2 H_c^{\vee}\otimes \bigwedge^2V^{\vee}$ which satisfies (A1)-(A3). Thus $A \in \cA[c,r]$ and there exists a monad
\begin{equation}\label{mon}
\xymatrix{\cM_A: H_c \otimes \cO_{\PP^n}(-1)\ar[r]^-{a_A}&  W \otimes \cO_{\PP^n}  \ar[rr]^-{a_A^{\vee}\circ (q_A\otimes \textrm{Id})} & &H^{\vee}_c \otimes \cO_{\PP^n}(1)},
\end{equation}

\nii whose cohomology sheaf is denoted by $\cE_A$. On the other hand, by (\cite{Jar2005} - Theorem 3), $\cE$ is cohomology of the monad
\begin{equation}\label{mon2}
\small{\xymatrix{0\ar[r]&\textrm{H}^1(\cE(1)\otimes \Omega_{\PP^3}^2)\otimes \cO_{\PP^n}(-1)\ar[r]^-{d_1^{-2,1}} & \textrm{H}^1(\cE\otimes \Omega_{\PP^n}^1)\otimes \cO_{\PP^n}\ar[r]^-{d_1^{-1,1}}&\textrm{H}^1(\cE(-1))\otimes \cO_{\PP^n}(1)\ar[r]& 0}}.
\end{equation}

\nii By the Lemma \ref{ctt} the monads \eqref{mon} and \eqref{mon2} are isomorphic. Thus $A$ defines a monad whose cohomology sheaf $\cE_A$ is isomorphic to $\cE$.\\

\nii Tensoring $\cM_A$ by $\cO_{\PP^n}(-n)$ and using \eqref{EAdef}, we obtain $\textrm{H}^{n-1}(\cE_A(-n))\cong \textrm{H}^{n}(H_c \otimes \cO_{\PP^n}(-n-1))$. Note that $\textrm{h}^{n}(H_c \otimes \cO_{\PP^n}(-n-1))=c$, then there exists $\xymatrix{f_A:H_c \ar[r]^-{\cong}& \textrm{H}^{n-1}(\cE_A(-n)) }$.\\

\nii Furthermore, the symmetric map $q_A$ induces a canonical isomorphism of monads

$$\xymatrix{\cM_A:\ar[d]_-{\Phi_A}& H_c \otimes \cO_{\PP^n}(-1)\ar[r]^-{a_A} \ar[d]_{\textrm{Id}}&  W \otimes \cO_{\PP^n}  \ar[d]^{q_a \otimes \textrm{Id}} \ar[rr]^-{a_A^{\vee}\circ (q_A\otimes \textrm{Id})} & &H^{\vee}_n \otimes \cO_{\PP^n}(1)\ar[d]_{\textrm{Id}}\\
\cM_A^{\vee}: &H_c \otimes \cO_{\PP^n}(-1)\ar[r]_-{(q_A\otimes \textrm{Id})\circ a_A}&  W^{\vee} \otimes \cO_{\PP^n}  \ar[rr]_-{a_A^{\vee}} & &H^{\vee}_c \otimes \cO_{\PP^n}(1)}
$$

\nii which induces a symmetric isomorphism of vector bundles $\xymatrix{\phi_A:\cE_A \ar[r]^-{\cong}& \cE_A^{\vee} }$.

\nii Thus, the data $[\cE_A,\phi_A,f_A]$ can be recovered from $A$.

\end{proof}

\nii By Theorem \ref{thmequivalence} the existence of orthogonal instanton bundles with charge $c$, rank $r$ and no global sections on $\PP^n$ is related to the existence of symmetric and non-degenerate linear maps. This approach is extremely helpful in the proof of the next result.   \\

\begin{pps}\label{mainthm} Let $c\geq 3$ be an integer. Every orthogonal instanton bundle $\cE$ on $\PP^n$ ($n\geq 3$) with no global sections and charge $c$ has rank $r\leq (n-1)c$. Moreover, there are no orthogonal instanton bundles $\cE$ on $\PP^n$ ($n\geq 3$) with no global sections and charge $c$ equal to $1$ or $2$.
\end{pps}

\begin{proof}
\nii First suppose that there exists an orthogonal instanton bundle $\cE$ with no global sections, charge $c$ and rank $r$ over $\PP^n$ and consider its equivalence class $[\cE,\phi, f]$. By Theorem \ref{thmequivalence} there exists $A \in \cA[c,r]$ associated with $[\cE,\phi, f]$.

\nii Given $A \in \cA[c,r]$, with some abuse of notation, let us also denote by $A$ the matrix associated with the morphism $\xymatrix{A:H_c\otimes V \ar[r]& H_c^{\vee}\otimes V^{\vee}}$. So $\textrm{rank }{A}=2c+r\leq (n+1)c,$ implying $\textrm{rank } \cE=r\leq (n-1)c$.

\nii To conclude the proof, if $c=1$, then $A \in\bigwedge^2 H_1^{\vee}\otimes \bigwedge^2V^{\vee} \cong 0$, but the zero map is degenerate.

\nii If $c=2$, then $A \in\bigwedge^2 H_2^{\vee}\otimes \bigwedge^2V^{\vee}\cong \KK \otimes \bigwedge^2V^{\vee}$, so
$A$ is skew-symmetric, but $A$ is also symmetric, hence $A$ is the zero map.

\nii So there are no orthogonal instanton bundles, with no global sections and charge $1$ or $2$ on $\PP^n$.\\

\end{proof}

\nii Now, with the help of Macaulay2, see \cite{M2}, we will construct explicit examples of orthogonal instanton bundles on $\PP^n$ when $r$ reaches the upper bound. Let us start by explaining the consequences of the results obtained.

\nii Proposition \ref{mainthm} and diagram \eqref{defqA} imply that $A\cong q_A$. Moreover, we have

\begin{equation*}
\xymatrix{a^{\vee}_A: W^{\vee} \otimes \cO_{\PP^3} \ar[r]^-{\textrm{Id}}&H^{\vee}_n \otimes V^{\vee}\otimes \cO_{\PP^3} \ar[r]^-{ev}& H^{\vee}_n \otimes \cO_{\PP^3}(1).}
\end{equation*}

\nii So if $\{x_0, x_1, \ldots, x_n \}$ is a basis of $V^{\vee}$, we have the monad

\begin{equation}\label{monadexe}
\xymatrix{\cM_A: H_c \otimes \cO_{\PP^n}(-1)\ar[r]^-{a_A}&  W \otimes \cO_{\PP^n}  \ar[rr]^-{a_A^{\vee}\circ (A\otimes \textrm{Id})} & &H^{\vee}_c \otimes \cO_{\PP^n}(1)},
\end{equation}

\nii where $a_A^{\vee}$ is given by

$$
\displaystyle a^{\vee}_A =\left( \tiny \begin{array}{cccccccccccccccccc}
x_0 & x_1 & \ldots & x_n &0&0&\ldots&0&&\cdots&&&&&&&&\\
0&0&\ldots&0&x_0 & x_1 & \ldots & x_n & &\cdots&&&&&&&&\\
&&&&&&&&&\ddots&&&&&&\\
&&&&&& &&& \cdots&0&0&\ldots&0&x_0 & x_1 & \ldots & x_n
\end{array}  \right)_{(n+1)\times(n+1)c.}
$$

\nii Theorem \ref{thmequivalence} simplifies the search for orthogonal instanton bundles and translates our existence problem in a linear algebra problem: we have to look for invertible matrices in $\bigwedge^2 H_c^{\vee} \otimes \bigwedge^2 V^{\vee}$. Recall that every skew-symmetric matrix $M$ can be written as a block diagonal matrix

\begin{equation} \label{skewmatrix}
\left(\begin{array}{cccccccccc}
0 &  & \lambda_1 &  &  & &  &  &  &  \\
 &  &  &  & 0 &  &  &  &  &  \\
-\lambda_1 &  & 0 &  &  &  & \cdots &  & 0 &  \\
 &  &  & 0 &  & \lambda_2 & &  &  &  \\
 & 0 &  &  &  &  &  &  & 0 &  \\
 &  &  & -\lambda_2 &  & 0 &  &  &  &  \\
 & \vdots &  & &  &  & \ddots &  & \vdots &  \\
 &  &  &  &  &  &  & 0 &  & \lambda_l \\
 & 0 &  &  & 0 &  & \cdots &  &  &  \\
 &  &  &  &  &  &  & -\lambda_l &  & 0
\end{array} \right)_{2l\times 2l}.
\end{equation}

\nii where $\pm i\lambda_i$ are the non-zero eigenvalues of $M$.
\nii In order to build examples of orthogonal instanton bundles with even charge $c$ on $\PP^n$, with $n$ odd, we can take two matrices $B$ and $C$ as in \eqref{skewmatrix}, where:

\begin{itemize}
\item $B$ is a $c \times c$ skew-symmetric matrix;
\item $C$ is a $(n+1) \times (n+1)$ skew-symmetric matrix.
\end{itemize}

\nii So if we consider $A=B \otimes C$, then $A \in \bigwedge^2 H_c^{\vee} \otimes \bigwedge^2 V^{\vee}$.

\begin{exm}\label{c6p3}
\nii Let us construct an example of orthogonal instanton bundle with no global sections and charge $6$ on $\PP^3$. Let $\{x_0,x_1,x_2,x_3\}$ be a basis for $V^{\vee}$. Consider \begin{center}
$B=\left( \begin{array}{cccccc}
0&2&0&0&0&0\\
-2&0&0&0&0&0\\
0&0&0&-1&0&0\\
0&0&1&0&0&0\\
0&0&0&0&0&1\\
0&0&0&0&-1&0
\end{array}
\right)$ and
$C=\left( \begin{array}{cccc}
0&1&0&0\\
-1&0&0&0\\
0&0&0&-3\\
0&0&3&0
\end{array}
\right).$
\end{center}

\nii Let $A= B\otimes C \in \bigwedge^2 H_6^{\vee} \otimes \bigwedge^2 V^{\vee}$. We have that $\textrm{rank } A=24$, so $A$ is invertible and therefore non degenerate. By Theorem \ref{thmequivalence} and Proposition \ref{mainthm} we have the linear monad

\begin{equation*}
\xymatrix{\cO_{\PP^3}^6(-1)\ar[r]^-{\alpha}& \cO_{\PP^3}^{24} \ar[r]^-{\beta}& \cO_{\PP^3}^6(1)
}
\end{equation*}

\nii where
\begin{center}
$\alpha=\left( \tiny \begin{array}{cccccc}
x_0 & 0 & 0& 0 & 0& 0\\
x_1 & 0 & 0 & 0& 0& 0\\
x_2 & 0 & 0 & 0& 0& 0\\
x_3 & 0 & 0& 0& 0& 0 \\
0 & x_0 & 0 & 0& 0& 0\\
0 & x_1 & 0 & 0& 0& 0\\
0 & x_2 & 0 & 0& 0& 0\\
0 & x_3 & 0 & 0& 0& 0\\
0 & 0 & x_0 & 0& 0& 0\\
0 & 0 & x_1 & 0& 0& 0\\
0 & 0 & x_2 & 0& 0& 0\\
0 & 0 & x_3& 0& 0& 0\\
0 &0 & 0 & x_0& 0& 0\\
0 &0 & 0 & x_1& 0& 0\\
0 &0 & 0 & x_2& 0& 0\\
0 &0 & 0 & x_3& 0& 0\\
0 &0 &0 & 0 & x_0 & 0\\
0 & 0 &0 &0 & x_1 & 0\\
0 & 0 &0 &0 & x_2 & 0\\
0 & 0 & 0 &0 &x_3& 0\\
0 &0 & 0 & 0 &0 &x_0\\
0 &0 & 0 & 0 &0 &x_1\\
0 &0 & 0 & 0 &0 &x_2\\
0 &0 & 0 & 0 &0 &x_3

\end{array} \right)$,
$\beta^t =\left( \tiny \begin{array}{cccccc}
0&2x_1&0&0&0&0\\
0&-2x_0&0&0&0&0\\
0&-6x_3&0&0&0&0\\
0&6x_2&0&0&0&0\\
-2x_1&0&0&0&0&0\\
2x_0&0&0&0&0&0\\
6x_3&0&0&0&0&0\\
-6x_2&0&0&0&0&0\\
0&0&0&-x_1&0&0\\
0&0&0&x_0&0&0\\
0&0&0&-3x_3&0&0\\
0&0&0&3x_2&0&0\\
0&0&x_1&0&0&0\\
0&0&x_0&0&0&0\\
0&0&-3x_3&0&0&0\\
0&0&3x_2&0&0&0\\
0&0&0&0&0&x_1\\
0&0&0&0&0&-x_0\\
0&0&0&0&0&-3x_3\\
0&0&0&0&0&3x_2\\
0&0&0&0&-x_1&0\\
0&0&0&0&x_0&0\\
0&0&0&0&3x_3&0\\
0&0&0&0&-3x_2&0\\

\end{array}  \right)$
\end{center}

\nii and whose cohomology bundle is an orthogonal instanton bundle $\cE$ on $\PP^3$ with no global sections, charge $6$ and rank $12$ .\\
\end{exm}

\nii As we can see in the next example, when $c$ is odd or $n$ is even, we need to be a little more careful, because skew-symmetric matrices of odd order do not have complete rank.

\begin{exm}\label{c5p3}
\nii For $c=5$ and $n=3$, we consider
\begin{center}
$B_1=\left( \begin{array}{ccccc}
0&1&0&0&0\\
-1&0&1&0&0\\
0&-1&0&0&0\\
0&0&0&0&0\\
0&0&0&0&0
\end{array}
\right)$
$B_2=\left( \begin{array}{ccccc}
0&0&1&0&1\\
0&0&0&0&0\\
-1&0&0&0&0\\
0&0&0&0&1\\
-1&0&0&-1&0
\end{array}
\right)$
$B_3=\left( \begin{array}{ccccc}
0&0&0&0&0\\
0&0&0&1&0\\
0&0&0&0&1\\
0&-1&0&0&0\\
0&0&-1&0&0
\end{array}
\right)$
\end{center}

\begin{center}
$C_1=\left( \begin{array}{cccc}
0&1&0&0\\
-1&0&0&0\\
0&0&0&1\\
0&0&-1&0
\end{array}
\right)$
$C_2=\left( \begin{array}{cccc}
0&0&0&1\\
0&0&1&0\\
0&-1&0&0\\
-1&0&0&0
\end{array}
\right)$
$C_3=\left( \begin{array}{cccc}
0&0&0&1\\
0&0&0&0\\
0&0&0&1\\
-1&0&-1&0
\end{array}
\right).$
\end{center}

\nii Let $A= B_1\otimes C_1 + B_2\otimes C_2 +B_3\otimes C_3$. We have that $\textrm{rank }A=20$, so $A$ is invertible and non degenerate. By Theorem \ref{thmequivalence} and Proposition \ref{mainthm} we have the linear monad

\begin{equation*}\label{charge5}
\xymatrix{\cO_{\PP^3}^5(-1)\ar[r]^-{\alpha}& \cO_{\PP^3}^{20} \ar[r]^-{\beta}& \cO_{\PP^3}^5(1)
}
\end{equation*}

\nii where
\begin{center}
$\alpha=\left( \tiny \begin{array}{ccccc}
x_0 & 0 & 0& 0 &0\\
x_1 & 0 & 0 & 0& 0\\
x_2 & 0 & 0 & 0& 0\\
x_3 & 0 & 0& 0 & 0\\
0 & x_0 & 0 & 0& 0\\
0 & x_1 & 0 & 0& 0\\
0 & x_2 & 0 & 0& 0\\
0 & x_3 & 0 & 0& 0\\
0 & 0 & x_0 & 0& 0\\
0 & 0 & x_1 & 0& 0\\
0 & 0 & x_2 & 0& 0\\
0 & 0 & x_3& 0& 0\\
0 &0 & 0 & x_0& 0\\
0 &0 & 0 & x_1& 0\\
0 &0 & 0 & x_2& 0\\
0 &0 & 0 & x_3& 0\\
0 &0 &0 & 0 & x_0\\
0 &0 &0 & 0 & x_1\\
0 &0 &0 & 0 & x_2\\
0 &0 &0 & 0 & x_3
\end{array} \right)$ and
$\beta^t =\left( \tiny \begin{array}{ccccc}
0&x_1&x_3&0&x_3\\
0&-x_0&x_2&0&x_2\\
0&x_3&-x_1&0&-x_1\\
0&-x_2&-x_0&0&-x_0\\
-x_1&0&x_1&x_3&0\\
x_0&0&-x_0&0&0\\
-x_3&0&x_3&x_3&0\\
x_2&0&-x_2&-x_0-x_2&0\\
-x_3&-x_1&0&0&x_3\\
x_2&x_0&0&0&0\\
x_1&-x_3&0&0&x_3\\
x_0&x_2&0&0&-x_0-x_2\\
0&-x_3&0&0&-x_3\\
0&0&0&0&x_2\\
0&-x_3&0&0&-x_1\\
0&x_0+x_2&0&0&-x_0\\
-x_3&0&-x_3&-x_3&0\\
-x_2&0&0&-x_2&0\\
-x_3&0&-x_3&0&0\\
x_0&0&x_0+x_2&x_0&0\\

\end{array}  \right)$
\end{center}

\nii as constructed before. The vector bundle which is the cohomology of the monad \eqref{charge5} is an orthogonal instanton bundle $\cE$ on $\PP^3$ with no global sections, charge $5$ and rank $10$.\\

\end{exm}

\section{Moduli space}

\nii  In this section we will keep focusing on orthogonal instanton bundles with maximal possible rank. Our goal is to use geometric invariant theory (GIT) to construct $\cM^{\cO}_{\PP^n}(c, r)$, the moduli space of orthogonal instanton bundles with charge $c$, rank $r$ and no global sections on $\PP^n$, for $n,c\geq 3$. First notice when $r=(n-1)c$, the conditions (A1) and (A3) are superfluous, and we have

$$\cA[c,(n-1)c]=\{ A \in \bigwedge^2 H_c^{\vee}\otimes \bigwedge^2  V^{\vee} ;  A \textrm{ is non degenerate}\}.$$

\nii Denote $\EE_c=\EE[c,(n-1)c]$, $\cA_c=\cA[c,(n-1)c]$, $G=GL(H_c)$, and let $\widetilde{\EE_c}$ be the set of isomorphism classes $[\cE,\phi]$ such that $[\cE,\phi,f]\in \EE_c$. Consider the action

\begin{equation*}
\begin{array}{ccl}
\alpha: G \times  \bigwedge^2 H_c^{\vee}\otimes \bigwedge^2  V^{\vee}& \rightarrow & \bigwedge^2 H_c^{\vee}\otimes \bigwedge^2  V^{\vee}\\
(h,A)& \mapsto & (h\otimes \textrm{Id})A(h^{\vee}\otimes \textrm{Id}).
\end{array}
\end{equation*}

\begin{lem}\label{ginv}
The set $\cA_c$ is $G$-invariant subset of $\bigwedge^2 H_c^{\vee}\otimes \bigwedge^2  V^{\vee}$.
\end{lem}
\begin{proof}
\nii Let $h\in G$, $A \in \bigwedge^2 H_c^{\vee}\otimes \bigwedge^2  V^{\vee}$ and $B=\alpha(h, A)$ the image of $h$ and $A$ by the previous action, that means
$$B=(h\otimes \textrm{Id})A(h^{\vee}\otimes \textrm{Id}).$$
\nii We can write $A= \sum_{i}(C_i\otimes D_i)$, where $C_i\in \bigwedge^2 H_c^{\vee}$ and $D_i\in \bigwedge^2 V^{\vee}$ for all integers $i$. Thus
\begin{equation*}
\begin{array}{rcl}
B&=&(h\otimes \textrm{Id})(\sum_{i}(C_i\otimes D_i))(h^{\vee}\otimes \textrm{Id})\\
&=&\sum_i((h  C_i h^{\vee})\otimes D_i).
\end{array}
\end{equation*}
\nii Since $h  C_i h^{\vee}\in \bigwedge^2 H_c^{\vee}$ for all integers $i$, it follows that $B \in \bigwedge^2 H_c^{\vee}\otimes \bigwedge^2  V^{\vee}$.
\end{proof}

\nii The bijection given in the next theorem is the key ingredient to construct $\cM_{\PP^n}^{\cO}(c)$.

\begin{thm}\label{orbit}
There is a bijection between the set of isomorphism classes $\widetilde{\EE_c}$ and the orbit space $\cA_c/G$. The isotropy group in each point is $\{\pm \textrm{Id}_{H_c}\}$.
\end{thm}
\begin{proof}
\nii Given $A \in \cA_c$ by Theorem \ref{thmequivalence} there exists $[\cE_A,\phi_A,f_A] \in \EE_c$ and we can define

$$ \begin{array}{rrcl}
\Psi:&\cA_c & \rightarrow& \widetilde{\EE_c}\\
&A& \mapsto & [\cE_A,\phi_A].
\end{array}
$$

\nii We will prove that $\Psi/G: A_c/G \rightarrow \widetilde{\EE_c}$ is a bijection. First note that $\Psi$ factors through $\cA_c/G$; indeed, consider $A,B \in \cA_c$ such that there exists $h \in G$ with $\alpha(h,A)=B$. We have the following commutative diagram

\begin{equation*}
\xymatrix{H_c\otimes V\ar[r]^-{A}\ar[d]_-{h\otimes Id}& H_c^{\vee}\otimes V^{\vee}\ar[d]^-{(h^{\vee})^{-1}\otimes Id}\\
H_c\otimes V\ar[r]_-{B}& H_c^{\vee}\otimes V^{\vee}.}
\end{equation*}

\nii Since $A,B \in \cA_c$, we have $A$ and $B$ invertible and by diagram \eqref{defqA}, we have the following commutative diagram,

\begin{equation*}
\xymatrix{H_c\otimes \cO_{\PP^n}(-1)\ar[r]^-{\textrm{Id}\otimes ev^{\vee}}\ar[d]_-{h\otimes \textrm{Id}}&W_A\ar[r]^-{A\otimes \textrm{Id}}\otimes \cO_{\PP^n} \ar[d]_-{h\otimes \textrm{Id}}& W_A^{\vee}\otimes \cO_{\PP^n}\ar[d]^-{(h^{\vee})^{-1}\otimes \textrm{Id}}\ar[r]^-{\textrm{Id}\otimes ev}&H_c^{\vee}\cO_{\PP^n}(1)\ar[d]^-{(h^{\vee})^{-1}\otimes \textrm{Id}}\\
H_c\otimes \cO_{\PP^n}(-1)\ar[r]_-{ev^{\vee}}&W_B\otimes \cO_{\PP^n}\ar[r]_-{B\otimes\textrm{Id}}& W_B^{\vee}\otimes \cO_{\PP^n}\ar[r]_-{ev}&H_c^{\vee}\cO_{\PP^n}(1),}
\end{equation*}

\nii so, we have the isomorphism of monads

\begin{equation*}
\xymatrix{\cM_A:&0\ar[r]& H_c\otimes\cO_{\PP^n}(-1) \ar[r]^-{a_A} \ar[d]_-{h \otimes \textrm{Id}}& W_A\otimes \cO_{\PP^n}\ar[r]^-{a_A^{\vee}\circ (A\otimes \textrm{Id})}\ar[d]^-{h\otimes \textrm{Id}}&H_c^{\vee}\otimes\cO_{\PP^n}(1)\ar[r]\ar[d]^-{(h^{\vee})^{-1}\otimes \textrm{Id}}&0\\
\cM_B:&0\ar[r]& H_c\otimes\cO_{\PP^n}(-1)\ar[r]_-{a_B}& W_B\otimes \cO_{\PP^n}\ar[r]_-{a_B^{\vee}\circ (B\otimes \textrm{Id})}&H_c^{\vee}\otimes\cO_{\PP^n}(1)\ar[r]&0.}
\end{equation*}

\nii Considering the cohomology of the monads $\cM_A$ and $\cM_B$, we get $\Psi(A)=[\cE_A,\phi_A]=[\cE_B, \phi_B]=\Psi(B)$ and we have the following commutative diagram

\begin{equation*}
\xymatrix{\cA_c \ar[rr]^-{\Psi}\ar[rd]_-{\pi}&&\widetilde{\EE_c}\\
&\cA_c/G\ar[ru]_-{\Psi/G}&}
\end{equation*}

\nii The projection $\pi $ is surjective by definition and we have by Theorem \ref{thmequivalence} and Proposition \ref{mainthm} that $\Psi$ is surjective as well. This implies that $\Psi/G$ is surjective.

\nii We need now to prove that $\Psi/G$ is injective. Indeed, let $A, B\in \cA_c$ such that $\Psi(A)=[\cE_A,\phi_A]=[\cE_B, \phi_B]=\Psi(B)$, we will show that there exists $h \in G$ such that $A=\alpha(h,B)$. If $[\cE_A,\phi_A]=[\cE_B, \phi_B]$, then by definition there exists an isomorphism $\xymatrix{g:\cE_A \ar[r]^-{\cong}& \cE_B }$ such that the following diagram is commutative

$$\xymatrix{\cE_A \ar[r]^-{\phi_A} \ar[d]_-{g} &\cE_A^{\vee}\\
\cE_B\ar[r]_-{\phi_B} & \cE_B^{\vee}\ar[u]_-{g^{\vee}}.
}$$

\nii Hence by Lemma \ref{ctt} we have the commutative diagram, which works in a more general case,

\begin{equation*}
{\tiny \xymatrix{H_c\otimes V\ar[r]\ar[d]^-{\cong}\ar@{.>}@/^0.8cm/[rrr]|{A}& W_A\ar[r]^-{q_A}\ar[d]^{\cong}& W_A^{\vee}\ar[r]\ar[d]^-{\cong}& H_c^{\vee}\otimes V^{\vee}\ar[r]\ar[d]^-{\cong}& \textrm{Ker } A^{\vee}\ar[r]\ar[d]^-{\cong}&0\\
H^{n-1}(\cE_A\otimes \Omega^{n}(1))\otimes V \ar[r]\ar[d]^-{g^{\ast}}& H^{n-1}(\cE_A\otimes \Omega^{n-1})\ar[d]^-{g^{\ast}}& H^1(\cE_A\otimes \Omega^{1})\ar[l]^-{\cong}\ar[r]\ar[d]^-{g^{\ast}}& H^1(\cE_A(-1))\otimes V^{\vee}\ar[r]\ar[d]^-{g^{\ast}}& H^1(\cE_A)\ar[r]\ar[d]^-{g^{\ast}}&0\\
H^{n-1}(\cE_B\otimes \Omega^{n}(1))\otimes V \ar[r]& H^{n-1}(\cE_B\otimes \Omega^{n-1})& H^1(\cE_B\otimes \Omega^{1})\ar[l]^-{\cong}\ar[r]& H^1(\cE_B(-1))\otimes V^{\vee}\ar[r]& H^1(\cE_B)\ar[r]&0\\
H_c\otimes V\ar[r]\ar[u]^-{\cong}\ar@{.>}@/_0.8cm/[rrr]|{B}& W_B\ar[r]_-{q_B}\ar[u]^{\cong}& W_B^{\vee}\ar[r]\ar[u]^-{\cong}& H_c^{\vee}\otimes V^{\vee}\ar[r]\ar[u]^-{\cong}& \textrm{Ker } B^{\vee}\ar[r]\ar[u]^-{\cong}&0,
} }
\end{equation*}

\nii where $g^{\ast}$ denotes the morphisms induced by $g$ on the cohomology groups, but recall that in our case $H_c \otimes V\cong W_A$ and $\textrm{Ker } A \cong 0$. Thus, the middle blocks are commutative and the commutativity of the top and bottom blocks follows from Lemma \ref{ctt}. Therefore, there exists $h\in G$ such that $B=\alpha(h,A)$.

\nii Finally, we will prove that the isotropy group is $\{ \pm \textrm{Id}_{H_c}\}$. Let $h \in G$ and $A \in \cA_c$, such that $ A = \alpha(h,A)$. By Theorem \ref{thmequivalence} we have $[\cE_A,\phi_A,f_A]=[\cE_{\alpha(h,A)},\phi_{\alpha(h,A)},f_{\alpha(h,A)}]$, since they come from the same symmetric map; hence there exists an isomorphism $\xymatrix{g:\cE_A\ar[r]^-{\cong}& \cE_{\alpha(h,A)}}$ such that the following diagrams commute

\begin{equation*}
\xymatrix{ \cE_A \ar[r]^-{\phi_A} \ar[d]_-{g} & \cE_A^{\vee}&  &H_c \ar[r]^-{f_A} \ar[d]_-{\lambda \textrm{Id}_c}& \textrm{H}^{n-1}(\cE_A(-n))\ar[d]^-{g_{\ast}}\\
\cE_{\alpha(h,A)} \ar[r]_-{\phi_{\alpha(h,A)}} & \cE_{\alpha(h,A)}^{\vee} \ar[u]_-{g^{\vee}}&  &H_c \ar[r]_-{f_{\alpha(h,A)}}& \textrm{H}^{n-1}(\cE_{\alpha(h,A)}(-n)),
}
\end{equation*}

\nii with $\lambda \in \{-1,1\}$. Thus
\begin{equation}\label{fa}
g^{\ast}\circ f_A= \pm f_{\alpha(h,A)}.
\end{equation}

\nii On the other hand, since $A=\alpha(h,A)$ by Lemma \ref{ctt} we have the commutative diagram
\begin{equation*}
{\tiny \xymatrix{H_c\otimes V\ar[r]\ar[d]^-{f_A}\ar@{.>}@/^0.8cm/[rrr]|{A}\ar@{.>}@/_0.8cm/[ddd]|{h}& W_A\ar[r]^-{q_A}\ar[d]& W_A^{\vee}\ar[r]\ar[d]& H_c^{\vee}\otimes V^{\vee}\ar[d]_-{(f_A^{\vee})^{-1}}\ar@{.>}@/^0.8cm/[ddd]|{(h^{\vee})^{-1}}\\
H^{n-1}(\cE_A\otimes \Omega^{n}(1))\otimes V \ar[r]\ar[d]^-{g^{\ast}}& H^{n-1}(\cE_A\otimes \Omega^{n-1})\ar[d]^-{g^{\ast}}& H^1(\cE_A\otimes \Omega^{1})\ar[l]^{\cong}\ar[r]\ar[d]^-{g^{\ast}}& H^1(\cE_A(-1))\otimes V^{\vee}\ar[d]_-{g^{\ast}}\\
H^{n-1}(\cE_{\alpha(h,A)}\otimes \Omega^{n}(1))\otimes V \ar[r]& H^{n-1}(\cE_{\alpha(h,A)}\otimes \Omega^{n-1})& H^1(\cE_{\alpha(h,A)}\otimes \Omega^{1})\ar[l]^-{\cong}\ar[r]& H^1(\cE_{\alpha(h,A)}(-1))\otimes V^{\vee}\\
H_c\otimes V\ar[r]\ar[u]_-{f_{\alpha(h,A)}}\ar@{.>}@/_0.8cm/[rrr]|{{\alpha(h,A)}}& W_B\ar[r]_-{q_{\alpha(h,A)}}\ar[u]& W_{\alpha(h,A)}^{\vee}\ar[r]\ar[u]& H_c^{\vee}\otimes V^{\vee}\ar[u]^-{(f_{\alpha(h,A)}^{\vee})^{-1}},
} }
\end{equation*}

\nii and therefore, looking at the left column we have
\begin{equation*}
\begin{array}{rcl}
h&=&(f_{\alpha(h,A)})^{-1}\circ g^{\ast}\circ f_A\\
&=&(f_{\alpha(h,A)})^{-1}\circ (\pm f_{\alpha(h,A)}) \textrm{   \hspace{1cm} (By \eqref{fa})}\\
&=&\pm \textrm{Id}_{H_c},
\end{array}
\end{equation*}

\nii and the isotropy group is $\{\pm \textrm{Id}_{H_c}\}.$

\end{proof}

\nii Since $G=GL(H_c)$ is a reductive group and its isotropy group $\{\pm\textrm{Id}_{H_c}\}$ is a discrete subgroup, the quotient $G_0=G/\{\pm\textrm{Id}_{H_c}\}$ is also reductive. Moreover, the action of $G_0$ on $\cA_c$ is free and we have:

\begin{thm}\label{finemoduli} The geometric quotient $\cM^{\cO}_{\PP^n}(c,(n-1)c):=\cA_c//G_0$ is reduced and irreducible affine coarse moduli space of dimension ${{c}\choose{2}}{{n+1}\choose{2}} -c^2$ for orthogonal instanton bundles with charge $c$, rank $(n-1)c$ and no global sections on $\PP^n$, for $n,c\geq 3$.
\end{thm}
\begin{proof}
First note that $\cA_c$ is an open dense subset of $\bigwedge^{2}H_c^{\vee} \otimes \bigwedge^2V^{\vee}$ which is affine, reduced and irreducible, thus $\cA_c$ is also affine, reduced and irreducible. Moreover being $G_0$ a reductive group, then $\cM^{\cO}_{\PP^n}(c,(n-1)c):=\cA_c//G_0$ is an affine good quotient and therefore $\cM^{\cO}_{\PP^n}(c,(n-1)c)$ is an affine, reduced and irreducible categorical quotient. Since the action is free all orbits of the action are closed and it follows that $\cM^{\cO}_{\PP^n}(c,(n-1)c)$ is an affine, reduced and irreducible coarse moduli space.

\nii The dimension of $\cM^{\cO}_{\PP^n}(c,(n-1)c)$ can be computed as
\begin{equation*}
\displaystyle \begin{array}{rcl}
\dim \cA_c - \dim G_0& =& \dim (\bigwedge^2 H_c^{\vee}\otimes \bigwedge^2  V^{\vee})- \dim G\\
&=& {{c}\choose{2}}{{n+1}\choose{2}} -c^2.
\end{array}
\end{equation*}

\end{proof}

\begin{rem}
Since $\cM^{\cO}_{\PP^n}(c,(n-1)c)$ is reduced and irreducible it follows that $\cM^{\cO}_{\PP^n}(c,(n-1)c)$ is generically smooth. But a question arises naturally in this context is:
	
\end{rem}

\begin{quest}
Is $\cM^{\cO}_{\PP^n}(c)$ smooth?
\end{quest}

\section{Splitting type}

\nii The goal of this section is to determine the type of splitting of orthogonal instanton bundles in $\cM_{\PP^n}^{\cO}(c,r)$ for $n,c\geq 3$, whenever non-empty.\\ 

\nii Jardim, Marchesi and Wi\ss{}dorf proved in (\cite{JMW2016} - Lemma 4.3 and Theorem 4.4) that there are no orthogonal instanton bundles of trivial splitting type, arbitrary rank $r$, and charge $2$ or odd on $\mathbb{P}^n$.  In order to determine the splitting type of the orthogonal instanton bundles, with no global section, charge $c$ and rank $r$ on $\PP^n$ we will associate these bundles to Kronecker modules.

\begin{defi}\label{kronecker}
A {\bf Kronecker module} of rank $r$ is a linear map
\begin{equation*}
\gamma: \bigwedge^2 V \rightarrow \textrm{Hom}(H_c, H_c^{\vee}),
\end{equation*}
\nii such that for the associated linear map,
\begin{equation*}
\hat{\gamma}: V \otimes H_c \rightarrow V^{\vee} \otimes H_c^{\vee},
\end{equation*}
\nii defined by $\hat{\gamma}(v_1 \otimes h_1)(v_2\otimes h_2)=[\gamma(v_1\wedge v_2)(h_1)](h_2)$, the following statements hold

\begin{itemize}
\item[(K1)] $\hat{\gamma}(v\otimes -): H_c \rightarrow V^{\vee} \otimes H_c^{\vee}$ is injective for all $v \neq 0$.
\item[(K2)]If $v^{\vee \vee}: V^{\vee}\otimes H_c^{\vee}\rightarrow H_c^{\vee}$ is the evaluation map associated to $v\in V$, then $v^{\vee \vee}\circ\hat{\gamma}: V\otimes H_c\rightarrow H_c^{\vee}$ is surjective for all $v \neq 0$.
\item[(K3)] $\rk \hat{\gamma}=2n+r$.
\end{itemize}

\end{defi}

\nii In section 3, we saw that given $[\cE, \phi , f] \in \mathbb{E}[c,r]$, with $n,c\geq 3$, by Theorem \ref{thmequivalence} there exists $A \in \cA[c,r]$ and the monad below

\begin{equation}\label{monadE}
\xymatrix{\cM_A: 0\ar[r]&H_c \otimes \cO_{\PP^n}(-1)\ar[r]^-{a_A}&  W \otimes \cO_{\PP^n}  \ar[rr]^-{a_A^{\vee}\circ (A\otimes \textrm{Id})} & &H^{\vee}_c \otimes \cO_{\PP^n}(1)\ar[r]&0},
\end{equation}

\nii whose cohomological bundle $\cE_A$ is isomorphic to $\cE$.

\nii Now let us use the maps $a_A$ and $b_A=a_A^{\vee}\circ (A\otimes \textrm{Id})$ in \eqref{monadE} to construct a Kronecker module associated to $\cE$. We can associate to $a_A$ and $b_A$ the linear maps $\alpha \in \textrm{Hom}(V,\textrm{Hom}(H_c,W))$ and $\beta \in \textrm{Hom}(V,\textrm{Hom}(W, H_c^{\vee}))$ as follows

\begin{equation*}
\begin{array}{rrll}
\alpha: V\rightarrow & \textrm{Hom}(H_c,W)&&\\
v\mapsto & \alpha(v):H_c & \rightarrow & W\\
& h & \mapsto & a_A(x)(h\otimes v)
\end{array}
\end{equation*}

\nii and

\begin{equation*}
\begin{array}{rrll}
\beta: V\rightarrow & \textrm{Hom}(W,H_c^{\vee})&&\\
v\mapsto & \beta(v):W & \rightarrow & H_c^{\vee}\\
& w & \mapsto & b_A(x)(w)(h\otimes v)
\end{array}
\end{equation*}

\nii where $x=\PP(\KK v)$. First note that $b_A(x)(w)(h\otimes v)= A(w)(\alpha(v)(h))$, this is why $\beta$ it is also known as the {\bf transpose map of $\alpha$} (with respect to $A$).

\nii This pair of maps $(\alpha, \beta)$ has the following properties:

\begin{itemize}
\item[(P1)] $\alpha(v):H_c\rightarrow W$ is injective for all $v\neq 0$;
\item[(P2)] $\beta(v) \circ \alpha(v):H_c \rightarrow H_c^{\vee}$ is the zero mapping for all $v\in V$;
\item[(P3)] the map $\hat{\alpha}:V\otimes H_c\rightarrow W$ is surjective, with $\hat{\alpha}(v\otimes h)=\alpha(v)(h)$.
\end{itemize}

\nii The property (P1) holds if and only if $a_A$ is injective in each fiber.

\nii The property (P2) is equivalent for the composition $a_A^{\vee}\circ(A \otimes \textrm{Id})\circ a_A$ to be the zero mapping in each fiber. Indeed, for each $x \in \PP(\KK v)\in \PP^n$, we have
\begin{equation*}
\begin{array}{rcl}
(b_A \circ a_A)(x)(h\otimes v)& = & [a_A^{\vee}\circ (A\otimes \textrm{Id})\circ a_A](x)(h\otimes v)\\
&=&A(\alpha(v)(h))(\alpha(v)(h))\\
&=&\beta(v) \circ \alpha(v).
\end{array}
\end{equation*}

\nii Now, let us prove that (P3) holds if and only if the cohomology bundle of \eqref{monadE} has no global sections. Indeed, by the display of the monad \eqref{monadE}, we have the following exact sequences

$$\xymatrix{0\ar[r]&H_c \otimes \cO_{\PP^n}(-1)\ar[r]^-{a_A}&  \textrm{Ker }(b_A)\ar[r] & \cE_A \ar[r]&0\\
0\ar[r]&\textrm{Ker }(b_A)\ar[r]&  W \otimes \cO_{\PP^n}  \ar[r]^-{b_A} &H^{\vee}_c \otimes \cO_{\PP^n}(1)\ar[r]&0,}
$$

\nii thus,
$$H^0(\cE_A)\cong H^0(\textrm{Ker }(b_A))\cong \textrm{Ker } (W\rightarrow H_c^{\vee}\otimes V^{\vee}),$$
\nii that means $H^0(\cE_A)=0$ if and only if $ H^0(b_A) :W\rightarrow H_c^{\vee}\otimes V^{\vee}$ is injective, if and only if $\hat{\alpha}:V\otimes H_c\rightarrow W$ is surjective.

\nii Let us consider

\begin{equation}\label{KronE}
\begin{array}{rcl}
\gamma^{\prime}: V\times V& \rightarrow& \textrm{Hom}(H_c, H_c^{\vee})\\
(v_1,v_2)& \mapsto& \beta(v_2)\circ \alpha(v_1),
\end{array}
\end{equation}

\nii which defines an element $\gamma \in \textrm{Hom}(\bigwedge^{2}V,\textrm{Hom}(H_c,H_c^{\vee}))$. We now will prove that the map $\gamma$ is a Kronecker module of rank $r$.

\begin{lem}
The element $\gamma \in \textrm{Hom}(\bigwedge^{2}V,\textrm{Hom}(H_c,H_c^{\vee}))$ constructed as above is a Kronecker module of rank $r$.
\end{lem}
\begin{proof}
Let $\hat{\alpha}:V\otimes H_c\rightarrow W$, $\hat{\beta}:W\rightarrow V^{\vee}\otimes H_c^{\vee}$ and $\hat{\gamma}:H_c \otimes V\rightarrow H_c^{\vee}\otimes V^{\vee}$ be the linear maps associated to $\alpha$, $\beta$ and $\gamma$, respectively. By the definition of $\gamma^{\prime}$ in \eqref{KronE} we have $\hat{\gamma}=\hat{\beta}\circ \hat{\alpha}$. Now let us prove that $\gamma$ satisfies the properties (K1)-(K3) of Definition \ref{kronecker}.

\nii For each $v\neq 0$ we have $$\hat{\gamma}(v\otimes h)[\gamma(v\wedge v_1)(h)](h_1),$$
\nii but

\begin{equation*}
\begin{array}{rcl}
[\gamma^{\prime}(v\wedge v_1)(h)](h_1)&=&[\beta(v_1)\circ\alpha(v)](h)(h_1)\\
&=&A(\alpha(v)(h))(\alpha(v_1)(h_1)).
\end{array}
\end{equation*}

\nii Thus (K1) follows by property (P1). Also observe that $v^{\vee \vee}\circ \hat{\gamma}=\hat{\gamma}(v\otimes -)^{\vee}$, therefore (K1) implies (K2).

\nii By (P3) we have that $\hat{\alpha}$ is surjective and $\hat{\beta}$ is injective, thus it follows that $\rk \hat{\gamma}= (n+1)c$, moreover, since $\hat{\gamma}=\hat{\beta}\circ \hat{\alpha}$, we have rank $\gamma=r$ and the property (P3) and therefore $\gamma$ is a Kronecker module of rank $r$.

\end{proof}

\begin{rem} Note that the linear map $\hat{\gamma}$ associated to the Kronecker module $\gamma$ is in fact the map $A$. Fixing the basis $\{v_1, \cdots , v_{n+1}\}$ and $\{h_1, \cdots , h_c\}$ for $V$ and $H_c$, repectively, for any $v_j \otimes h_i \in V\otimes H_c$, we have $\hat{\gamma}(v_j\otimes h_i)= A(v_j\otimes h_i)$. Indeed, for any $v_k \otimes h_l \in V\otimes H_c$ it follows

\begin{equation*}
\begin{array}{rcl}
\hat{\gamma}(v_j\otimes h_i)(v_k\otimes h_l)&=& (\hat{\beta}\circ \hat{\alpha})(v_j\otimes h_i)(v_k\otimes h_l)\\
&=&\hat{\beta}(\hat{\alpha} (v_j\otimes h_i)) (v_k\otimes h_l)\\
&=& \hat{\beta}(\alpha (v_j)(h_i)) (v_k\otimes h_l)\\
&=& \hat{\beta}(a_A (x_j)(h_i\otimes v_j)) (v_k\otimes h_l)\\
&=& \hat{\beta}(h_i\otimes v_j) (v_k\otimes h_l)\\
&=& \beta(v_k)(h_i\otimes v_j)) (h_l)\\
&=& b_A(x_k)(h_i\otimes v_j) (h_l\otimes v_k)\\
&=& A(h_i\otimes v_j) (\alpha (v_k) (h_l))\\
&=& A(h_i\otimes v_j) (h_l\otimes v_k),
\end{array}
\end{equation*}

\noindent where $x_j=\PP(\KK v_j)$ and $x_k=\PP(\KK v_k)$.

\end{rem}

 The following result describes how we can obtain the splitting type of an orthogonal instanton  bundle. 

\begin{thm}\label{splitting}
Let $\cE$ be an orthogonal instanton bundle on $\PP^n$ with charge $c$, rank $r$ and no global sections, for $n,c\geq 3$, and let $\gamma$ be its associated Kronecker module. If $L \subset\PP^n$ is the line defined by $v_1,v_2\in V$, $v_1\wedge v_2\neq 0$, the restriction $\cE|_L$ is trivial if and only if $\gamma(v_1\wedge v_2)$ is an isomorphism.

\end{thm}
\begin{proof}
Let $\cE$ be an orthogonal instanton bundle on $\PP^n$ with charge $c$, rank $r$ and no global sections, for $n,c\geq 3$. Consider the maps $(\alpha, \beta)$ and the Kronecker module $\gamma$ as before. Let $v_1,v_2 \in V$ such that $v_1 \wedge v_2\neq 0$, and consider the $\KK$-subspace $K=\KK v_1+ \KK v_2$. The restriction of the monad \eqref{monadE} to $L=\PP(K)$ is the monad

\begin{equation}\label{monadL}
\xymatrix{\cM_{A}: 0\ar[r]&H_c \otimes \cO_{L}(-1)\ar[r]^-{{a_A}|_L}&  W \otimes \cO_{L}  \ar[r]^-{{b_A|_L}}  &H^{\vee}_c \otimes \cO_{L}(1)\ar[r]&0.}
\end{equation}

\nii The display of the monad \eqref{monadL} gives the exact sequences

\begin{equation*}
\xymatrix{0\ar[r]&H_c \otimes \cO_{L}(-1)\ar[r]^-{{a_A}|_L}&  \textrm{Ker } ({b_A|_L})\ar[r] & \cE|_{L}\ar[r]&0,}
\end{equation*}

\begin{equation*}
\xymatrix{0\ar[r]&\textrm{Ker } ({b_A|_L})\ar[r]&  W \otimes \cO_{L}  \ar[r]^-{b_A|_L} &H^{\vee}_c \otimes \cO_{L}(1)\ar[r]&0,}
\end{equation*}

\nii thus

$$
H^0(L, \cE|_L)\cong H^0(L, \textrm{Ker }( b_A|_L))\cong \textrm{Ker }( W \rightarrow H_c^{\vee} \otimes K^{\vee}).
$$
\nii Observe that $\cE|_L$ has trivial splitting type if and only if no section $s \in H^0(L, \cE|_L)\setminus \{0\}$ has zeros, so our goal is to prove that this holds if and only if $\gamma(v_1\wedge v_2)$ is invertible. Consider the inclusions

\begin{equation*}
H_c\otimes \cO_L(-1)\stackrel{i}{\hookrightarrow} \textrm{Ker }(b_A|_L) \stackrel{j}{\hookrightarrow} W \otimes \cO_L.
\end{equation*}

\nii Let $s \in H^0(L, \textrm{Ker }(b_A|_L)))\cong H^0(L, \cE|_L)$ be a section; being $H^0(W\otimes \cO_L)\cong W$, there exists $w \in W$ with $j \circ s(x)=w$ for all $x \in L$. So the section $s^{\prime} \in H^0(L, \cE|_L)$ defined by $s$ has zeros at $x = \PP(\KK v)\in L$ if and only if $s(x)$ lies in the image of the inclusion $i(x):H_c\otimes \cO_L(-1) \hookrightarrow \textrm{Ker }(b_A|_L)(x)$, i.e. if and only if there exists $h \in H_c$ with $\alpha(v)(h)=w$. Because $s$ is a section in $ \textrm{Ker }(b_A|_L)$, for every $v^{\prime}\in K$ we must have $\beta(v^{\prime})(w)=0$, thus $\cE|_L$ has no trivial section with zeros if and only if
$$\textrm{Im } \alpha(v)\subset \bigcap_{v^{\prime}\in K}\textrm{Ker } \beta (v^{\prime}), $$

\nii for at least one vector $v \in K \setminus \{0\}$, which means that for any basis $v, v^{\prime} \in K$ of $K$ the map
$$\gamma(v\wedge v^{\prime})=\beta(v^{\prime})\circ \alpha(v)$$
\nii is not an isomorphism.
\end{proof}

\nii Now let us use the Theorem \ref{splitting} to determine the type of splitting of the bundle constructed in Example \ref{c6p3}.

\begin{exm}
\nii Let $\cE$ be the orthogonal instanton bundle on $\PP^3$ of Example \ref{c6p3}. Let $L\in \PP^3$ be the line joining two general points $P=[a:b:c:d]$ and $Q=[e:f:g:h]$. By Theorem \ref{splitting}, $\cE|_L$ is trivial if and only if $\beta(Q)\alpha(P)$ is invertible. In our case,

\begin{equation*}
\beta(Q)\alpha(P)=\left(\begin{array}{cccccc}
0&\lambda_1&0&0&0&0\\
-\lambda_1&0&0&0&0&0\\
0&0&0&-\lambda_2&0&0\\
0&0&\lambda_2&0&0&0\\
0&0&0&0&0&\lambda_3\\
0&0&0&0&-\lambda_3&0
\end{array}\right)
\end{equation*}
\nii where $\lambda_1=2be-2af-6dg+6ch$, $\lambda_2=-be+af+3dg-3ch$ and $\lambda_3=be-af-3dg+3ch$. Thus $\beta(Q)\alpha(P)$ is invertible and therefore $\cE|_L$ is trivial. Since the points are general this is also true for every general line, therefore $\cE$ has trivial splitting type.\\
\nii On the other hand, let $L_0$ be the line joining the points $P=[1:0:0:0]$ and $Q=[0:0:0:1]$. By the previous construction $\beta(Q)\alpha(P)=0$, therefore by Theorem \ref{splitting} $\cE|_{L_0}$ is not trivial, hence $L_0$ is a jumping line for $\cE$.
\end{exm}

\nii Finally, let us prove that the bundle presented in Example \ref{c5p3} has no trivial splitting type, as expected from (\cite{JMW2016} - Lemma 4.3).

\begin{exm}
\nii Let $\cE$ be the orthogonal instanton bundle on $\PP^3$ Example \ref{c5p3}. Let $L\in \PP^3$ be the line joint the points $P=[a:b:c:d]$ and $Q=[e:f:g:h]$. We have

\begin{equation*}
\beta(Q)\alpha(P)=\left(\begin{array}{ccccc}
0&\lambda_1&\lambda_2&0&\lambda_2\\
-\lambda_1&0&\lambda_1&\lambda_2&0\\
-\lambda_2&-\lambda_1&0&0&\lambda_2\\
0&-\lambda_2&0&0&\lambda_2\\
-\lambda_2&0&-\lambda_2&-\lambda_2&0
\end{array}\right)
\end{equation*}
\nii where $\lambda_1=be-af+dg-ch$ and $\lambda_2=de+cf-bg-ah$. Since $\beta(Q)\alpha(P)$ is a skew-symmetric matrix of odd order, $\beta(Q)\alpha(P)$ is not invertible for all $P$ and $Q$.  Thus by Theorem \ref{splitting}, $\cE|_L$ is not trivial for every line $L\in \PP^3$, i.e. $\cE$ has no trivial splitting type.

\end{exm}

\end{document}